\documentclass[a4paper,11pt]{amsart}

\usepackage[utf8]{inputenc}
\usepackage[toc]{appendix}

\usepackage[english]{babel}

\usepackage{graphicx}
\usepackage{subfigure}
\usepackage[justification=centering]{caption}

\usepackage{amsmath,a4wide}
\usepackage{amsfonts}
\usepackage{amssymb}
\usepackage{amsthm}
\usepackage{mathtools}
\usepackage{amsaddr}

\usepackage{enumerate}
\usepackage{IEEEtrantools}

\usepackage{hyperref}
\usepackage{url}

\numberwithin{equation}{section}
\mathtoolsset{showonlyrefs}

\providecommand{\norm}[1]{\big\lVert#1\big\rVert}

\newcommand{\R}{\mathbb{R}}
\newcommand{\C}{\mathbb{C}}

\newcommand{\Z}{\mathbb{Z}}
\newcommand{\N}{\mathbb{N}}

\newcommand{\T}{\mathbb{T}}

\renewcommand{\l}{\lambda}
\renewcommand{\L}{\Lambda}

\newcommand{\detp}{{\det}^\prime}
\newcommand{\sump}[1]{{\sum_{#1}}^\prime}
\newcommand{\tr}{\textnormal{tr}}

\theoremstyle{plain}
\newtheorem{theorem}{Theorem}[section]

\theoremstyle{plain}

\theoremstyle{plain}
\newtheorem{lemma}[theorem]{Lemma}

\theoremstyle{plain}
\newtheorem{proposition}[theorem]{Proposition}

\theoremstyle{definition}
\newtheorem{definition}[theorem]{Definition}

\theoremstyle{remark}

\theoremstyle{remark}

\theoremstyle{definition}

\usepackage{xcolor}

\begin{document}
\title[Extremal Determinants of Laplace--Beltrami Operators for Rectangular Tori]{Extremal Determinants of Laplace--Beltrami Operators for Rectangular Tori}
\author[Markus Faulhuber]{Markus Faulhuber}
\address{NuHAG, Faculty of Mathematics, University of Vienna\\Oskar-Morgenstern-Platz 1, 1090 Vienna, Austria}
\email{markus.faulhuber@univie.ac.at}
\thanks{The author was partially supported by the Vienna Science and Technology Fund (WWTF): VRG12-009 and by an Erwin-Schrödinger fellowship of the Austrian Science Fund (FWF): J4100-N32. This work was partially established during the first stage of the fellowship, which the author spent with the Analysis Group at NTNU Trondheim, Norway. The computational results presented have been achieved (in part) using the Vienna Scientific Cluster (VSC). The author wishes to thank the anonymous referee for a very careful reading and catching some typos. Also, the referee caught a flaw in the first version of the proof of Proposition \ref{prop_log_log_concave}.}

\begin{abstract}
	In this work we study the determinant of the Laplace--Beltrami operator on rectangular tori of unit area. We will see that the square torus gives the extremal determinant within this class of tori. The result is established by studying properties of the Dedekind eta function for special arguments. Refined logarithmic convexity and concavity results of the classical Jacobi theta functions of one real variable are deeply involved.
\end{abstract}


\maketitle

\section{Introduction}\label{sec_Intro}

The search for extremal geometries is a popular topic in many branches of mathematics and mathematical physics. In this work, we pick up a result by Osgood, Phillips and Sarnak \cite{Osgood_Determinants_1988} on extremals of determinants of Laplace--Beltrami operators on tori and restrict the assumptions, excluding their solution of the following problem.

Among all 2-dimensional tori of area 1, which torus maximizes the determinant of the Laplace-Beltrami operator?

The answer in \cite{Osgood_Determinants_1988} is that the torus identified with the plane modulo a hexagonal (sometimes called triangular or equilateral) lattice gives the unique solution. However, if we only consider rectangular lattices, this solution is not possible and the natural assumption is that the square lattice will lead to the optimal solution. We will prove that this is indeed the case. Both problems, the one for general and the one for rectangular lattices, are closely related to the study of extremal values of the heat kernel on the torus \cite{Baernstein_HeatKernel_1997}, \cite{BaernsteinVinson_Local_1998}, finding extremal bounds of Gaussian Gabor frames of given density \cite{Faulhuber_PhD_2016}, \cite{Faulhuber_Hexagonal_2018}, \cite{FaulhuberSteinerberger_Theta_2017}, as well as the study of certain theta functions \cite{Montgomery_Theta_1988}.

It is worth noting that in all cases the extremal solutions are the same as for the classical sphere packing and covering problem in the plane. This immediately raises the question about extremal solutions for the above problems in higher dimensions. An interesting aspect is that in higher dimensions the optimal arrangements for the sphere packing and covering problem differ from each other. Hence, in order to guess what the right solution could be, one first has to decide whether one deals with a packing or a covering problem. However, we will not discuss higher dimensions in this work.

Another common theme is that we deal with theta functions in one way or another, which take a prominent role in several branches of mathematics. They appear in the studies on energy minimization \cite{BeterminPetrache_DimensionReduction_2017}, the study of Riemann's zeta and xi function \cite{CoffeyCsordas_2013}, \cite{Csordas_2015} or the theory of sphere packing and covering \cite{ConSlo98}, including the recent breakthrough for sphere packings in dimension 8 by Viazovska \cite{Viazovska8_2017} and in dimension 24 \cite{Viazovska24_2017}. They also appear in the field of time-frequency analysis and the study of Gaussian Gabor frames \cite{Faulhuber_Hexagonal_2018}, \cite{FaulhuberSteinerberger_Theta_2017}, \cite{Jan96} or the study of the heat kernel \cite{SteSha_Complex_03} to name just a few.

An open question concerns how the above problems can be linked to old, unsolved problems in geometric function theory, namely finding the exact values of Bloch's constant (1925) \cite{Blo25} and Landau's constant (1929) \cite{Lan29}. The correct solutions are conjectured to be given in the work of Ahlfors and Grunsky \cite{AhlGru37} and in the work of \footnote{In an unpublished work, Robinson came up with the same solution as Rademacher in 1937.}{Rademacher} \cite{Rad43}, respectively.

Baernstein repeatedly suggested that a better understanding of the behavior of extremal values of the heat kernel on the torus might lead to new insights for the mentioned constants \cite{Baernstein_ExtremalProblems_1994}, \cite{Baernstein_HeatKernel_1997}, \cite{BaernsteinVinson_Local_1998} and, after some research, the author shares this opinion. In fact, this work is a result of the author's study on a conjecture of Baernstein, Eremenko, Frytnov and Solynin \cite{Baernstein_Metric_2005} related to Landau's constant, which Eremenko posed again as an open problem in an unpublished preprint in 2011 \cite{Eremenko_Hyperbolic_2011}. However, besides the fact that in both cases varying metrics on rectangular tori are involved, it is not clear to the author how deep the connection between this work and the mentioned conjecture in \cite{Baernstein_Metric_2005} and \cite{Eremenko_Hyperbolic_2011} truly is.

Let us return to the question posed at the beginning. In \cite{Osgood_Determinants_1988} we find the following result, which is \footnote{In \cite{Osgood_Determinants_1988} there is a typo concerning the spanning vectors of the extremal lattice, it should read $\L = \sqrt{\tfrac{2}{\sqrt{3}}} \left\langle 1, \tfrac{1+i\sqrt{3}}{2} \right\rangle_\Z$.}{Corollary 1.3(b)} in that work. For a lattice $\L$, let $\Delta_\L$ be the Laplace-Beltrami operator for the torus $\T_\L  = \C/\L$ of unit area, then
\begin{equation}
	\detp \Delta_\L \leq \tfrac{\sqrt{3}}{2} \, | \eta ( \tfrac{1 + i\sqrt{3}}{2} ) |^4 \approx 0.35575 \dots
\end{equation}
with equality if and only if $\L$ is hexagonal. The result was established by first showing that the hexagonal lattice gives a local maximum by exploiting general facts about modular forms. Then, a numerical check gave the result that this local maximum is indeed global.

In another work, Sarnak \cite{Sarnak_Extremal_1997} mentions that \footnote{Peter Sarnak does not mention any reference for this claim. Anton Karnaukh was his PhD student around that time, but the author also could not find any result in that direction in Karnaukh's thesis \cite{Karnaukh_Doctoral_1996}.}{Karnaukh} has shown that the square lattice gives the only other critical point and that it is a saddle point. In particular, this implies that there exists a one--parameter family of lattices, within which the square lattice yields a local maximum of the determinant. Actually, we will see that among all rectangular tori, the square torus gives the global maximum. Our main result is as follows.

\begin{theorem}[Main Result]\label{thm_main}
	For $\alpha \in \R_+$, we denote the rectangular torus of unit area by $\T_\alpha = \C/(\alpha^{-1/2} \Z \times i \,   \alpha^{1/2} \Z)$ and the Laplace-Beltrami operator by $\Delta_\alpha$. Then
	\begin{equation}
		\detp \Delta_\alpha = \alpha \, | \eta(\alpha i) |^4 \leq | \eta(i) |^4 \approx 0.34830 \dots
	\end{equation}
	with equality if and only if $\alpha = 1$.
\end{theorem}
Our proof of Theorem \ref{thm_main} will at first be parallel to the proof in \cite{Osgood_Determinants_1988}, in particular, we will show that the problem about the determinant can be transferred to a problem of finding the maximum of the Dedekind eta function on a ray in the upper half plane. After that point, the proof will differ greatly from the methods in \cite{Osgood_Determinants_1988}. Not only will we show that the square lattice yields the global maximum, we will also give a precise behavior of the determinant as the lattice parameter varies. The key in our proof is to exploit the fact that the eta function can be decomposed into a product of Jacobi's theta functions. The result will follow from certain logarithmic convexity and concavity results, partially established already in \cite{FaulhuberSteinerberger_Theta_2017}.

This work is structured as follows;
\begin{itemize}
	\item In Section \ref{sec_torus} we recall the definitions of Laplace--Beltrami operators on tori and their determinants as well as the results from the work of Osgood, Phillips and Sarnak \cite{Osgood_Determinants_1988}. Also, we will see how the determinant connects with the Dedekind eta function.
	
	\item In Section \ref{sec_Jacobi} we define Jacobi's classical theta functions of one real variable and show how they can be used to express the Dedekind eta function. The proof of the main result will follow from refined logarithmic convexity and concavity statements related to Jacobi's theta functions as described by Faulhuber and Steinerberger \cite{FaulhuberSteinerberger_Theta_2017}.
\end{itemize}

\section{The Laplace--Beltrami Operator on the Torus}\label{sec_torus}

In this section we recall the results established in \cite{Osgood_Determinants_1988} and how the problem of finding extremal surfaces for determinants of Laplace--Beltrami operators connects with the Dedekind eta function. Osgood, Phillips and Sarnak studied the determinant of the Laplace--Beltrami operator on a surface with varying metric as a function of the metric. In the case of the torus with flat metric, the varying of the metric can be interpreted as varying the lattice associated to the torus (and keeping the standard metric). Furthermore, we will introduce the heat kernel of a Laplace--Beltrami operator and its determinant.

We denote the Laplace--Beltrami operator on the 2-dimensional torus $\T_\L = \C/\L$ by $\Delta_\L$. The torus is identified with the fundamental domain of the lattice $\L$ which is a discrete subgroup of $\C$. A lattice $\L$ is generated by integer linear combinations of two nonzero complex numbers $z_1$ and $z_2$ with the property that $\frac{z_1}{z_2} \notin \R$;
\begin{equation}
	\L = \langle z_1, z_2 \rangle_\Z = \{m z_1 + n z_2 \mid m,n \in \Z, \, z_1, z_2 \in \C, \, \tfrac{z_1}{z_2} \notin \R \}.
\end{equation}
The area of the torus is then defined to be the area of a fundamental domain, i.e.,
\begin{equation}
	\textnormal{area}(\L) = | Im(\overline{z_1} z_2)| = |x_1 y_2 - x_2 y_1|, \qquad z_k = x_k + i \, y_k, \, k = 1,2.
\end{equation}
In this work we solely deal with rectangular lattices, i.e., we can choose a basis of the lattice $\langle z_1, z_2 \rangle_\Z$ with the property that the ratio of $z_1$ and $z_2$ is purely imaginary;
\begin{equation}
	\L \textnormal{ rectangular} \quad \Longleftrightarrow \quad \exists z_1, z_2 \in \C \colon \L = \langle z_1, z_2 \rangle_\Z \; \wedge \; i \frac{z_1}{z_2} \in \R.
\end{equation}
We note that any rectangular lattice can be identified with $\alpha \Z \times i \, \beta \Z$ where $\alpha, \beta \in \R_+$. In this work, it will be no restriction to assume that the lattice has unit area, i.e., $\alpha \beta = 1$.

\subsection{The Laplace--Beltrami Operator on Manifolds and its Heat Kernel}
We will now introduce Laplace--Beltrami operators on connected Riemannian manifolds as well as the associated heat kernel and the determinant. For further reading on heat kernels we refer to \cite{Jorgenson_Heat_2001}. Also, the procedure of introducing determinants of Laplace--Beltrami operators is described in \cite{BaernsteinVinson_Local_1998}, \cite{Osgood_Determinants_1988} or \cite{Sarnak_Determinants_1990}, and we will follow these references.

Just for the moment, let us change to a more general notation. Let $M$ be a connected Riemannian manifold and $\Delta_M$ the corresponding Laplace--Beltrami operator. The heat semigroup is defined as
\begin{equation}
	P_t = \{ e^{t \Delta_M} \mid t \in \R_+ \}.
\end{equation}
The action of the group $P_t$ on a function $f \in L^2(M)$ is (by abuse of notation) given by the integral operator
\begin{equation}
	P_t f(x) = \int_M p_t(x,y) f(y) \, d \mu(y),
\end{equation}
where $\mu(y)$ is the \footnote{$(M, \mu)$ can also be a weighted manifold with $\mu$ being any measure with smooth positive density with respect to the Riemannian measure.}{Lebesgue measure on $M$}. The integral kernel $p_t$ is called the heat kernel of the Laplace--Beltrami operator. For any $y \in M$, the heat kernel fulfills the heat equation
\begin{equation}
	\Delta_M u - \partial_t u = 0
\end{equation}
and for any $y \in M$
\begin{equation}
	p_t(\, . \, , \, y) \to \delta_y, \qquad t \to 0.
\end{equation}
If the spectrum of $\Delta_M$ is discrete and consists of eigenvalues $\{\l_k\}_{k=1}^\infty$ with $\{\phi_k\}_{k=1}^\infty$ being the sequence of corresponding eigenfunctions, constituting an orthonormal basis for $L^2(M)$, then the heat kernel can be expanded as
\begin{equation}
	p_t(x,y) = \sum_{k=1}^\infty e^{-\l_k t} \phi_k(x) \overline{\phi_k(y)}.
\end{equation}
Also, in this case the trace of the heat kernel is given by
\begin{equation}
	\tr(e^{t \Delta_M}) = \int_M p_t(x,x) \, d \mu(x) = \sum_{k=1}^\infty e^{-\l_k t}.
\end{equation}
The left--hand side of the above equation defines the trace of the heat kernel in general. To the Laplace--Beltrami operator one associates a zeta function in the following way
\begin{equation}
	Z_M(s) = \sum_{k=1}^\infty \l_k^{-s}, \qquad Re(s) > 1.
\end{equation}
This function can be continued analytically in $\C \backslash \{ 1 \}$. Formally, the determinant of $\Delta_M$ is given by
\begin{equation}
	\det \Delta_M = \prod_{\l_k \neq 0} \l_k.
\end{equation}
This product is not necessarily meaningful and the proper definition of the determinant is to use the zeta regularization
\begin{equation}
	\detp \Delta_M = e^{\left. -\tfrac{d}{ds} Z_M \right\rvert_{s=0}}.
\end{equation}
A closely related function is the height function of a Riemannian manifold \cite{Sarnak_Determinants_1990},
\begin{equation}
	h_M = - \log \detp \Delta_M = \left. \dfrac{d}{ds} Z_M \right\rvert_{s=0}.
\end{equation}
This function is an isospectral invariant and it is clear that problems about $\detp \Delta_M$ can be transferred to problems about $h_M$ and vice versa.

\subsection{The Determinant on the Torus}
We return to the case of the torus $\C/\L$ where $\L$ is a lattice of unit area. As we will see, it is no restriction to assume that $\L = c \, \langle 1, z\rangle_\Z$, where $c$ is chosen such that the area of the lattice is 1, i.e., $c = Im(z)^{-1/2}$. In other words, the problem under consideration is invariant under rotation and scaling, just like the classical sphere packing and covering problems.

In this case, the eigenvalues of the Laplace--Beltrami operator $\Delta_\L$ are $(2 \pi |z_\l|)^2$, where $z_\l \in \L$. At this point, we mention that the eigenvalues are actually given by $(2 \pi |z_\l^\bot|)^2$, $z_\l^\bot \in \L^\bot$, the dual lattice. However, for 2-dimensional lattices, the relation between a lattice and its dual lattice is simply given by
\begin{equation}
	\L^\bot = \textnormal{area}(\L)^{-1} \, i \, \L.
\end{equation}
This means that the dual lattice is a 90 degrees rotated, scaled version of the original lattice. However, as mentioned we deal with lattices of unit area, hence the scaling factor is irrelevant, as is the rotation. Thus, in our concrete situation, there is no need to distinguish between the lattice $\L$ and its dual $\L^\bot$ (this only results in a re--labeling of the eigenvalues). The zeta function is, consequently, given by
\begin{equation}
	Z_\L(s) = \sump{z_\l \in \L} (2 \pi | z_\l |)^{-2s},
\end{equation}
where the prime indicates that the sum does not include the origin. Using the definition of the lattice $\L = \langle z_1, z_2 \rangle_\Z$, we re-write the zeta function as
\begin{equation}
	Z_\L(s) = (2 \pi)^{-2s} \sump{k,l \in \Z} \frac{1}{|k z_1 + l z_2|^{2s}} = (2 \pi)^{-2s} \sump{k,l \in \Z} \frac{y^s}{|k + l z|^{2s}}.
\end{equation}
For the last equality we set $z = \tfrac{z_2}{z_1}$ and $y = Im(z) > 0$, where the second condition is imposed by the fact that the lattice has unit area. Osgood, Phillips and Sarnak \cite{Osgood_Determinants_1988} now use the fact that the last series is a multiple of the Eistenstein series
\begin{equation}
	E_\L(z,s) = \sump{k,l \in \Z} \frac{y^s}{|k + l z|^{2s}},
\end{equation}
hence,
\begin{equation}
	Z_\L(s) = (2 \pi)^{-2 s} E_\L(z,s).
\end{equation}
The final step in order to compute $\detp \Delta_\L$ is now to differentiate $Z_\L$ with respect to $s$ and evaluate at 0. We compute that
\begin{equation}
	\dfrac{d}{ds} Z_\L(s) = (2 \pi)^{-2 s} \left(-2 \log(2 \pi) \, E_\L(z,s) + \dfrac{\partial}{\partial s} E(z,s) \right).
\end{equation}
Now, the following result is a consequence of Kronecker's limit formula.
\begin{align}
	E_\L(z,0) & = -1,\\
	\left. \dfrac{\partial}{\partial s} E_\L \right\rvert_{s=0} & = - 2 \log\left( 2 \pi \, y^{1/2} | \eta(z) |^2 \right).
\end{align}
The Dedekind eta function $\eta(\tau)$ is defined in the upper half plane $\mathbb{H}$ by the infinite product
\begin{equation}\label{eq_eta}
	\eta(\tau) = e^{\pi i \tau/12} \prod_{k=1}^\infty (1-e^{2 \pi i k \tau}).
\end{equation}
Hence,
\begin{equation}
	\left. \dfrac{d}{ds} Z_\L \right\rvert_{s=0} = - 2 \log\left( y^{1/2} | \eta(z) |^2 \right).
\end{equation}
It follows that the (zeta regularized) determinant of the Laplace--Beltrami operator $\Delta_\L$, with $\L = y^{-1/2} \langle 1, z \rangle_\Z$, $z = x + i y$, $y>0$, is given by
\begin{equation}
	\detp \Delta_\L = y |\eta(z)|^4.
\end{equation}

\subsection{Rectangular Tori}
We have seen that maximizing the determinant $\detp \Delta_\L$, where the lattice is given by $\L = y^{-1/2} \langle 1, z \rangle_\Z$ corresponds to maximizing $y |\eta(z)|^4$. In the case of a rectangular torus $\C/(y^{-1/2} \Z \times i \, y^{1/2} \Z)$, the problem in focus is
\begin{equation}
	\textnormal{maximize} \quad y |\eta(i y)|^4, \qquad y \in \R_+.
\end{equation}
We could end this work now with the fact, already observed by Baernstein and Vinson \cite{BaernsteinVinson_Local_1998}, that maximizing $\detp \Delta_\L$ is implied by minimizing $\tr(e^{t \Delta_\L})$. Montgomery showed that the unique minimizer of $\tr(e^{t \Delta_\L})$ is the hexagonal lattice \cite{Montgomery_Theta_1988}. In a recent work on Gaussian Gabor frames, the results of Faulhuber and Steinerberger \cite{FaulhuberSteinerberger_Theta_2017} imply that the square lattice is the unique minimizer of $\tr(e^{t \Delta_\L})$ within the class of rectangluar lattices. Hence, Theorem \ref{thm_main} is implied by the results in \cite{FaulhuberSteinerberger_Theta_2017}. However, we will come up with a proof independent from the results in \cite{FaulhuberSteinerberger_Theta_2017} and \cite{Montgomery_Theta_1988}, but, however, the techniques are similar. The rest of this work is meant to give a deeper insight into the mentioned problems and maybe these insights can be valuable for some of the related problems mentioned in the introduction.

\section{Jacobi's Theta Functions}\label{sec_Jacobi}
In this section we study properties of Jacobi's theta functions which will lead to a deeper understanding of Theorem \ref{thm_main}. We start by defining the theta function in accordance with the textbook of Stein and Shakarchi \cite{SteSha_Complex_03}.
\begin{definition}
	For $z \in \C$ and $\tau \in \mathbb{H}$ (the upper half plane) we define the theta function as
	\begin{equation}\label{eq_Theta}
		\Theta(z, \tau) = \sum_{k \in \Z} e^{\pi i k^2 \tau} e^{2 k \pi i z}.
	\end{equation}
\end{definition}
This function is an entire function with respect to $z$ and holomorphic with respect to $\tau$. As stated in \cite{SteSha_Complex_03}, the function arises in many different fields of mathematics, such as the theory of elliptic functions, the theory of modular functions, as a fundamental solution of the heat equation on the torus as well as in the study of Riemann's zeta function. Also, it is used to prove results in combinatorics and number theory.

The function can also be expressed as an infinite product.

\begin{proposition}[Jacobi triple product]\label{prop_triple_product}\index{Jacobi triple product}
For $z \in \C$ and $\tau \in \mathbb{H}$ we have
	\begin{equation}\label{eq_triple_product}
		\Theta(z,\tau) = \prod_{k \geq 1} \left(1 - e^{2 k \pi i \tau}\right) \left(1 + e^{(2k-1) \pi i \tau} e^{2\pi i z}\right) \left(1 + e^{(2k-1) \pi i \tau} e^{-2\pi i z}\right).
	\end{equation}
\end{proposition}

It also fulfills the following identity.

\begin{theorem}\label{thm_Theta}
	For $z \in \C$ and $\tau \in \mathbb{H}$ we have
	\begin{equation}
		\Theta\left(z, -\tfrac{1}{\tau}\right) = (-i \tau)^{1/2} e^{\pi i z^2 \tau} \Theta(\tau z, \tau).
	\end{equation}
\end{theorem}
For a proof of Theorem \ref{thm_Theta} and more details on theta functions of two complex variables as well as the product representation we refer to the textbook of Stein and Shakarchi \cite{SteSha_Complex_03}.

Whereas derivatives of $\Theta$ with respect to $z$ are often studied, it seems that studies of its derivatives with respect to $\tau$ are less common. The following Lemma contains a symmetry result for the logarithmic derivative of $\Theta$ with respect to $\tau$.
\begin{lemma}
	For $z \in \C$ and $\tau \in \mathbb{H}$ we have
	\begin{equation}\label{lem_identity_Theta}
		\pi i z^2 \tau + \tau \frac{z \, \partial_z \Theta(\tau z, \tau) + \partial_\tau \Theta(\tau z, \tau)}{\Theta(\tau z, \tau)} - \frac{1}{\tau} \frac{\partial_\tau \Theta\left(z, -\tfrac{1}{\tau}\right)}{\Theta\left(z, -\tfrac{1}{\tau}\right)} = -\frac{1}{2}.
	\end{equation}
	In particular, for $z = 0$ we get
	\begin{equation}
		\tau \frac{\partial_\tau \Theta(0, \tau)}{\Theta(0, \tau)} - \frac{1}{\tau} \frac{\partial_\tau \Theta\left(0, -\tfrac{1}{\tau}\right)}{\Theta\left(0, -\tfrac{1}{\tau}\right)} = -\frac{1}{2}.
	\end{equation}
\end{lemma}
\begin{proof}
	We start by taking the logarithm on both sides of the identity in Theorem \ref{thm_Theta}
	\begin{equation}
		\log\left(\Theta\left(z, -\tfrac{1}{\tau}\right)\right) = \frac{1}{2} \log(-i \tau) + \pi i z^2 \tau +\log\left(\Theta(\tau z, \tau)\right).
	\end{equation}
	Differentiating with respect to $\tau$ on both sides and a multiplication by $\tau$ yields, after rearranging the terms, the desired result.
\end{proof}
As a next step, we define Jacobi's theta functions of two variables which we will then restrict to certain domains. For $z \in \C$ and $\tau \in \mathbb{H}$ we define
\begin{align}
	\vartheta_1(z,\tau) & = -i \sum_{k \in \Z} (-1)^{k} e^{\pi i (k+1/2)^2 \tau} e^{(2k+1)\pi i z}\\
		& = 2 \sum_{k \in \N} (-1)^{k} e^{\pi i (k+1/2)^2 \tau} \sin((2k+1)\pi z)\\
	\vartheta_2(z,\tau) & = \sum_{k \in \Z} e^{\pi i (k-1/2)^2 \tau} e^{(2k+1)\pi i z}\\
		& = 2 \sum_{k \in \N} e^{\pi i (k-1/2)^2 \tau} \cos((2k+1)\pi z)\\
	\vartheta_3(z,\tau) & = \sum_{k \in \Z} e^{\pi i k^2 \tau} e^{2 k \pi i z}\\
		& = 1 + 2 \sum_{k \in \N} e^{\pi i k^2 \tau} \cos(2 k \pi z)\\
	\vartheta_4(z,\tau) & = \sum_{k \in \Z} (-1)^k e^{\pi i k^2 \tau} e^{2 k \pi i z}\\
		& = 1 + 2 \sum_{k \in \N} (-1)^k e^{\pi i k^2 \tau} \cos(2 k \pi z).
\end{align}
We note that $\vartheta_3(z,\tau) = \Theta(z,\tau)$ and that any $\vartheta_j$ can be expressed via $\Theta$ from equation \eqref{eq_Theta}. The functions which we will study in the rest of this work are restrictions of the above functions with $(z,\tau) = (0,i x)$, $x \in\R_+$. In particular, all these functions are real-valued.
\begin{definition}\label{def_Jacobi_theta}
	For $x \in \R_+$ we define the real-valued theta functions in the following way.
	\begin{IEEEeqnarray}{rClClCl}
		\theta_2(x) = \vartheta_2(0,i x) & = & \sum_{k \in \Z} e^{-\pi (k-1/2)^2 x} & = & & & 2 \sum_{k \in \N} e^{-\pi (k-1/2)^2 x}\label{eq_def_theta_2},\\
		\theta_3(x) = \vartheta_3(0,i x)& = & \sum_{k \in \Z} e^{-\pi k^2 x} & = & 1 & + & 2 \sum_{k \in \N} e^{-\pi k^2 x}\label{eq_def_theta_3},\\
		\theta_4(x) = \vartheta_4(0,i x)& = & \sum_{k \in \Z} (-1)^k e^{-\pi k^2 x} & = & 1 & + & 2\sum_{k \in \N} (-1)^k e^{-\pi k^2 x}\label{eq_def_theta_4}.
	\end{IEEEeqnarray}
\end{definition}
It does not make much sense to study properties of $\theta_1(x)$ if defined as above, because $\vartheta_1(0,i x) = 0$ for all $x \in \R_+$. However, we will see that $\theta_1$ is involved in the proof of our main result in some sense.

All of the above functions can also be expressed by infinite products.
\begin{align}
	\theta_2(x) & = 2 e^{-\tfrac{\pi}{4} x} \prod_{k \in \N} \left(1 - e^{-2 k \pi x}\right)\left(1 + e^{-2 k \pi x}\right)^2\\
	\theta_3(x) & = \prod_{k \in \N} \left(1 - e^{-2 k \pi x}\right)\left(1 + e^{-(2 k - 1) \pi x}\right)^2\\
	\theta_4(x) & = \prod_{k \in \N} \left(1 - e^{-2 k \pi x}\right)\left(1 - e^{-(2 k - 1) \pi x}\right)^2
\end{align}
These representations can be quite useful when studying the logarithmic derivatives of these functions.
We note that for $z \in \R$ and purely imaginary $\tau = i x$, $x \in \R_+$ the theta function $\Theta(z, ix)$ is maximal for $z \in \Z$ and minimal for $z \in \Z + \frac{1}{2}$. These special cases correspond to Jacobi's $\theta_3$ and $\theta_4$ function
\begin{align}
	\theta_3(x) &= \Theta\left( 0,ix \right),\\
	\theta_4(x) &= \Theta\left( \tfrac{1}{2}, ix \right).
\end{align}
For $x \in \R_+$, the $\theta_2$-function can be expressed via $\Theta$ in the following way.
\begin{equation}
	\theta_2(x) = e^{-\pi \frac{x}{4}} \, \Theta \left( \tfrac{i x}{2},ix \right).
\end{equation}

A recurring theme will be the frequent use of the differential operator $x \, \frac{d}{dx}$. We will use the notation $\left(x \frac{d}{dx}\right)^n$ for its iterated repetition, i.e.,
\begin{equation}
	\left(x \frac{d}{dx}\right)^n = x \frac{d}{dx} \left(x \frac{d}{dx}\right)^{n-1}.
\end{equation}
In particular, we will study properties of the logarithmic derivative of a certain function $f: \R_+ \to \R_+$, on a logarithmic scale. Let us explain how this statement should be interpreted. By using the variable transformation $y = \log(x)$ we extend the domain from $\R_+$ to $\R$ and consider the new function $\log\left(f \left( e^y \right)\right)$. Therefore, we get an extra exponential factor each time we take a (logarithmic) derivative. We have
\begin{equation}
	\frac{d}{dy} \log\left(f \left( e^y \right)\right) = e^y \frac{f'\left(e^y\right)}{f\left(e^y\right)}.
\end{equation}
By reversing the transformation of variables, we come back to the original scale, but the factor $x$ stays. Without being explicitly mentioned, these methods were used in \cite{Faulhuber_Hexagonal_2018}, \cite{FaulhuberSteinerberger_Theta_2017}, \cite{Montgomery_Theta_1988} to establish uniqueness results about extremal theta functions on lattices. We will now provide a new point of view on as well as new results related to the mentioned articles.

For what follows it is necessary to clarify the notation of the Fourier transform and Poisson's summation formula which are given by
\begin{equation}\label{eq_FT}
	\widehat{f}(\omega) = \int_{\R} f(x) e^{-2 \pi i \omega x} \, dx, \qquad x,\omega \in \R
\end{equation}
and
\begin{equation}\label{eq_Poisson}
	\sum_{k \in \Z} f(k+x) = \sum_{l \in \Z} \widehat{f}(l) e^{2 \pi i l x}, \qquad x \in \R.
\end{equation}
respectively. Both formulas certainly hold for Schwartz functions and since we will apply both only on Gaussians we do not have to worry about more general properties for the formulas to hold.

As a consequence of the Poisson summation formula we find the following, well-known identities
\begin{equation}\label{eq_Poisson_theta_3}
	\sqrt{x} \, \theta_3 \left( x \right) = \theta_3 \left( \tfrac{1}{x} \right),
\end{equation}
\begin{equation}\label{eq_Poisson_theta_2_4}
	\sqrt{x} \, \theta_2 \left( x \right) = \theta_4 \left( \tfrac{1}{x} \right) \quad \text{ and } \quad
	\sqrt{x} \, \theta_4 \left( x \right) = \theta_2 \left( \tfrac{1}{x} \right).
\end{equation}

We note the common theme in the last identities which leads to the following lemma.
\begin{lemma}\label{lem_identity}
	Let $r \in \R$ and suppose that $f,g \in C^1\left( \R_+ , \R \right)$ do not possess zeros. If $f$ and $g$ satisfy the (generalized Jacobi) identity
	\begin{equation}\label{eq_Jacobi_general}
		x^r \, f\left( x \right) = g\left( \tfrac{1}{x} \right),
	\end{equation}
	then
	\begin{equation}\label{eq_identity}
		x \, \frac{f' \left( x \right)}{f \left( x \right)} + \tfrac{1}{x} \, \frac{g' \left( \tfrac{1}{x} \right)}{g \left( \tfrac{1}{x} \right)} = -r.
	\end{equation}
\end{lemma}
\begin{proof}
	Both, $f$ and $g$ are either positive or negative on $\R_+$ since they are real-valued, continuous and do not contain zeros. As they also fulfill the identity $x^r \, f(x) = g\left(\tfrac{1}{x}\right)$ both of them possess the same sign. Without loss of generality we may therefore assume that both functions are positive because otherwise we change the sign which does not affect the results. Therefore, we have
	\begin{equation}
		\log \left( x^r \, f(x) \right) = \log \left( g \left( \tfrac{1}{x} \right) \right).
	\end{equation}
	Using the differential operator $x \frac{d}{dx}$ on both sides directly yields
	\begin{equation}
		r + x \, \frac{f'(x)}{f(x)} = -\tfrac{1}{x} \, \frac{g'\left(\tfrac{1}{x}\right)}{g\left(\tfrac{1}{x}\right)}.
	\end{equation}
\end{proof}

We remark that an alternative proof of Lemma \ref{lem_identity} is given in \cite{FaulhuberSteinerberger_Theta_2017} for the special case $f = g = \theta_3$ and $r = \tfrac{1}{2}$ (a generalization of the proof in \cite{FaulhuberSteinerberger_Theta_2017} is possible with the arguments given there). With this proof, the result
\begin{equation}\label{eq_identity_theta_3}
	x \, \frac{\theta_3'(x)}{\theta_3(x)} + \tfrac{1}{x} \, \frac{\theta_3' \left( \tfrac{1}{x} \right)}{\theta_3 \left( \tfrac{1}{x} \right)} = -\frac{1}{2}
\end{equation}
was derived.

However, it seemed to have gone unnoticed that the result can be put into a more general context. With the previous lemma we also get the results
\begin{equation}\label{eq_identity_theta_2_4}
	x \, \frac{\theta_2'(x)}{\theta_2(x)} + \tfrac{1}{x} \, \frac{\theta_4' \left( \tfrac{1}{x} \right)}{\theta_4 \left( \tfrac{1}{x} \right)} = -\frac{1}{2} \qquad \text{ and } \qquad
	x \, \frac{\theta_4'(x)}{\theta_4(x)} + \tfrac{1}{x} \, \frac{\theta_2' \left( \tfrac{1}{x} \right)}{\theta_2 \left( \tfrac{1}{x} \right)} = -\frac{1}{2},
\end{equation}
which were not given in \cite{FaulhuberSteinerberger_Theta_2017}. Also, the author just learned that the above formulas can also be found in the monograph of Borwein and Borwein \cite[chap.~2.3]{Borwein_AGM_1987}. All of the above results actually contain symmetry statements about logarithmic derivatives of Jacobi's theta functions on a logarithmic scale. We note that the identities \eqref{eq_identity_theta_3} and \eqref{eq_identity_theta_2_4} are special cases of Lemma \ref{lem_identity_Theta} when $\Theta$ is restricted to certain rays in $\C \times \mathbb{H}$.

\subsection{Properties of Theta-2 and Theta-4}\label{sec_Theta_2_4}

We start with monotonicity properties of the logarithmic derivatives of Jacobi's $\theta_2$ and $\theta_4$ functions on a logarithmic scale. The following result was already given in \cite{FaulhuberSteinerberger_Theta_2017} and the proof can be established by using the product representation of $\theta_4$.
\begin{proposition}\label{prop_theta4_monotone}
	The function $x \frac{\theta'_4(x)}{\theta_4(x)}$ is strictly decreasing on $\R_+$.
\end{proposition}
As a consequence of Proposition \ref{prop_theta4_monotone} we derive the following result.
\begin{proposition}\label{prop_theta4_log_concave}
	The Jacobi theta function $\theta_4$ is strictly logarithmically concave or, equivalently, $\frac{\theta_4'(x)}{\theta_4(x)}$ is strictly decreasing. Also, the expression $\frac{\theta_4'(x)}{\theta_4(x)}$ is positive.
\end{proposition}
\begin{proof}
	The statement that $x \frac{\theta'_4(x)}{\theta_4(x)}$ is strictly decreasing on $\R_+$ was already proved in \cite{FaulhuberSteinerberger_Theta_2017}. We observe that $\lim_{x \to \infty} x \frac{\theta_4'(x)}{\theta_4(x)} = 0$, hence, $x \frac{\theta_4'(x)}{\theta_4(x)} > 0$ and, therefore, $\frac{\theta_4'(x)}{\theta_4(x)} > 0$. Also,
	\begin{equation}
		\frac{d}{dx} \left( x \frac{\theta_4'(x)}{\theta_4(x)} \right) = \frac{\theta_4'(x)}{\theta_4(x)} + x \frac{d}{dx} \left( \frac{\theta_4'(x)}{\theta_4(x)} \right) < 0,
	\end{equation}
	and it follows that
	\begin{equation}
		0 < \tfrac{1}{x} \frac{\theta_4'(x)}{\theta_4(x)} < - \frac{d}{dx} \left( \frac{\theta_4'(x)}{\theta_4(x)} \right),
	\end{equation}
	which proves that $\frac{\theta_4'(x)}{\theta_4(x)}$ is strictly decreasing, hence, $\theta_4$ is logarithmically concave.
\end{proof}

As a consequence of Proposition \ref{prop_theta4_monotone} and Lemma \ref{lem_identity} we obtain the following property of $\theta_2$, which was already claimed, but not proven, in \cite{FaulhuberSteinerberger_Theta_2017}.

\begin{proposition}\label{prop_theta2_monotone}
	The function $x \frac{\theta_2'(x)}{\theta_2(x)}$ is strictly decreasing on $\R_+$.
\end{proposition}
\begin{proof}
	Proposition \ref{prop_theta4_monotone} tells us that $x \frac{\theta_4'(x)}{\theta_4(x)}$ is strictly decreasing on $\R_+$. By using Lemma \ref{lem_identity} or the first identity in \eqref{eq_identity_theta_2_4} we obtain
	\begin{equation}
		\frac{d}{dx} \left(x \frac{\theta_2'(x)}{\theta_2(x)}\right) =
		\frac{d}{dx} \left( - \tfrac{1}{x} \frac{\theta_4' \left( \tfrac{1}{x} \right)}{\theta_4 \left( \tfrac{1}{x} \right)} -\frac{1}{2}\right) < 0.
	\end{equation}
\end{proof}
In fact, Proposition \ref{prop_theta4_monotone} and Proposition \ref{prop_theta2_monotone} are equivalent by Lemma \ref{lem_identity}. In \cite[Lemma 6.2]{FaulhuberSteinerberger_Theta_2017} Proposition \ref{prop_theta2_monotone} was proved for $x > 1$ by directly estimating $x \frac{\theta_2'(x)}{\theta_2(x)}$, which is a hard task for small $x$. The use of Lemma \ref{lem_identity} allows us to argue directly that $x \frac{\theta_2'(x)}{\theta_2(x)}$ and $x \frac{\theta_4'(x)}{\theta_4(x)}$ must possess the same monotonicity properties.
\begin{proposition}\label{prop_theta_2_log_convex}
	The Jacobi theta function $\theta_2$ is strictly logarithmically convex or, equivalently, $\frac{\theta_2'(x)}{\theta_2(x)}$ is strictly increasing. Also, the expression $\frac{\theta_2'(x)}{\theta_2(x)}$ is negative.
\end{proposition}
\begin{proof}
	The fact that $\frac{\theta_2'(x)}{\theta_2(x)} < 0$ is easily checked by using \eqref{eq_def_theta_2}. From the definition we see that $\theta_2 > 0$ and due to its unconditional convergence, the derivatives can be computed by differentiating each term. We find out that $\theta_2'(x) < 0$ and the negativity is proven.
	
	The logarithmic convexity statement basically follows from the Cauchy-Schwarz inequality for the Hilbert--space $\ell^2(\N)$. We compute
	\begin{equation}
		\frac{d^2}{dx^2} \Big(\log\big(\theta_2(x)\big)\Big) = \frac{\theta_2''(x) \theta_2(x) - \theta_2'(x)^2}{\theta_2(x)^2}.
	\end{equation}
	Strict logarithmic convexity of $\theta_2$ is now equivalent to
	\begin{equation}
		\theta_2''(x) \theta_2(x) - \theta_2'(x)^2 > 0,
	\end{equation}
	which can be re-written as
	\begin{equation}\label{eq_CSI_theta2}
		\left(\sum_{k \geq 1} \pi^2 \left(k-\tfrac{1}{2}\right)^4 e^{-\pi \left(k-\tfrac{1}{2}\right)^2 x}\right) \left(\sum_{k \geq 1} e^{-\pi \left(k-\tfrac{1}{2}\right)^2 x}\right) >
		\left(\sum_{k \geq 1} \pi \left(k-\tfrac{1}{2}\right)^2 e^{-\pi \left(k-\tfrac{1}{2}\right)^2 x}\right)^2.
	\end{equation}
	We set
	\begin{equation}
		\big(a_k\big)_{k = 1}^\infty = \left(\pi \left(k-\tfrac{1}{2}\right)^2 e^{-\pi \left(k-\tfrac{1}{2}\right)^2 x/2} \right)_{k = 1}^\infty
		\quad \text{ and } \quad
		\big(b_k\big)_{k = 1}^\infty = \left(e^{-\pi \left(k-\tfrac{1}{2}\right)^2 x/2} \right)_{k = 1}^\infty \, .
	\end{equation}
	For $x > 0$ fixed, $(a_k), (b_k) \in \ell^2(\N)$. Inequality \eqref{eq_CSI_theta2} is now equivalent to
	\begin{equation}
		\norm{(a_k)}_{\ell^2(\N)}^2 \norm{(b_k)}_{\ell^2(\N)}^2 > \big\langle (a_k), (b_k) \big\rangle_{\ell^2(\N)}^2,
	\end{equation}
	where strict inequality follows since $(a_k)$ and $(b_k)$ are linearly independent in $\ell^2(\N)$.
\end{proof}

Also, in \cite{FaulhuberSteinerberger_Theta_2017} it was claimed that $x^2 \frac{\theta_4'(x)}{\theta_4(x)}$ is monotonically decreasing and convex. Both properties were recently proved by Ernvall--Hytönen and Vesalainen \cite{HytVes_Theta4_2017}. We have a similar statement for $\theta_2$.

\begin{proposition}\label{prop_theta_2_4_monotone_s2}
	The functions $x^2 \frac{\theta_2'(x)}{\theta_2(x)}$ and $x^2 \frac{\theta_4'(x)}{\theta_4(x)}$ are strictly decreasing on $\R_+$.
\end{proposition}
\begin{proof}
	The result involving $\theta_4$ was proved by Ernvall--Hytönen and Vesalainen \cite{HytVes_Theta4_2017}.
	
	We proceed with proving the result involving $\theta_2$. We already know that $\theta_4$ is logarithmically concave, which is equivalent to the fact that $\frac{\theta_4'(x)}{\theta_4(x)}$ is monotonically decreasing. By using identity \eqref{eq_identity_theta_2_4} we get
	\begin{equation}
		x^2 \frac{\theta_2'(x)}{\theta_2(x)} = -\frac{x}{2} - \frac{\theta_4'(\tfrac{1}{x})}{\theta_4(\tfrac{1}{x})}.
	\end{equation}
	Therefore, we have
	\begin{equation}
		\frac{d}{dx} \left(x^2 \frac{\theta_2'(x)}{\theta_2(x)}\right) = -\frac{1}{2} - \, \underbrace{\frac{d}{dx} \left(\frac{\theta_4'(\tfrac{1}{x})}{\theta_4(\tfrac{1}{x})}\right)}_{> 0} < 0.
	\end{equation}
\end{proof}
We note that a similar simple argument as used for the expression involving $\theta_2$ does not work for $\theta_4$. Due to identity \eqref{eq_identity_theta_2_4} it also seems plausible to assume that for a monomial weight $x^r$, the second power, i.e.~$x^2$, is the limit for the monotonicity properties of the logarithmic derivative of $\theta_2$ and $\theta_4$ as stated in Proposition \ref{prop_theta_2_4_monotone_s2} and numerical investigations point in that direction. 

Using the fact that $x^2 \frac{\theta_4'(x)}{\theta_4(x)}$ is strictly decreasing and Proposition \ref{prop_theta_2_log_convex} we conclude that
\begin{equation}
	0 < x^2 \frac{d^2}{dx^2} \log(\theta_2(x)) = x^2 \frac{d}{dx} \left(\frac{\theta_2'(x)}{\theta_2(x)}\right)< \frac{1}{2}
\end{equation}	
Moreover, it seems to be true that $x^2 \frac{d^2}{dx^2} \log(\theta_2(x))$ is strictly decreasing, which would in return imply that $x^2 \frac{\theta_4'(x)}{\theta_4(x)}$ is strictly decreasing.

However, we leave this problem open and close this section with symmetry properties involving $\theta_2$, $\theta_4$ and repeated applications of $x \frac{d}{dx}$.
\begin{proposition}
	For $x \in \R_+$ and $2 \leq n \in \N$ the following holds.
	\begin{equation}
		\left(x \frac{d}{dx}\right)^n \log \left( \theta_2(x) \right) = (-1)^n \left(x \frac{d}{dx}\right)^n \log \left( \theta_4\left( \tfrac{1}{x} \right) \right).
	\end{equation}
\end{proposition}
\begin{proof}
	The statement is an obvious consequence of identity \eqref{eq_identity_theta_2_4} and follows by induction.
\end{proof}

\subsection{Properties of Theta-3}\label{sec_Theta_3}
We will now study properties of the logarithm of $\theta_3$ on a logarithmic scale.  
\begin{proposition}\label{prop_theta3_monotone}
	The function $x \frac{\theta_3'(x)}{\theta_3(x)}$ is strictly increasing on $\R_+$.
\end{proposition}
This result was proved in \cite{FaulhuberSteinerberger_Theta_2017}. As a consequence of Lemma \ref{lem_identity} we get the following results which were already proved in the author's doctoral thesis \cite{Faulhuber_PhD_2016}.
\begin{proposition}\label{pro_theta3_symmetry}
	For $x \in \R_+$ and $2 \leq n \in \N$ the following holds.
	\begin{equation}
		\left(x \frac{d}{dx}\right)^n \log \left( \theta_3(x) \right) = (-1)^n \left(x \frac{d}{dx}\right)^n \log \left( \theta_3\left( \tfrac{1}{x} \right) \right).
	\end{equation}
\end{proposition}
\begin{proof}
	The statement is an obvious consequence of identity \eqref{eq_identity_theta_3} and follows by induction.
\end{proof}
The following statement and its proof were also already given in \cite{Faulhuber_PhD_2016}. The techniques are similar to the techniques used by Montgomery in \cite{Montgomery_Theta_1988}. The computer algebra system Mathematica \cite{Mathematica} was used at some points in the proof in order to compute explicit values or closed expressions for geometric series, but in principle all computations can also be checked by hand. The proof was adjusted to a level which should be quite accessible with no or only little help of computer algebra software.
\begin{proposition}\label{prop_log_log_concave}
	For $x \in \R_+$, the function
	\begin{equation}\label{eq_log_log_concanve_prop}
		\left(x \frac{d}{dx}\right)^3 \log \left( \theta_3(x) \right)
	\end{equation}
	is positive for $x \in (0,1)$ and negative for $x > 1$. Also, the function is anti--symmetric in the following sense
	\begin{equation}
		\left(x \frac{d}{dx} \right)^3 \log \left( \theta_3(x) \right) = -\left(x \frac{d}{dx}\right)^3 \log \left( \theta_3\left( \tfrac{1}{x} \right) \right).
	\end{equation}
\end{proposition}
\begin{proof}
	To simplify notation we set
	\begin{equation}
		\psi(x) = (\log \circ \, \theta_3)'(x) = \frac{\theta_3'(x)}{\theta_3(x)}.
	\end{equation}
	Proposition \ref{pro_theta3_symmetry} settles the part concerning the anti--symmetry of expression \eqref{eq_log_log_concanve_prop}. Therefore, it is also clear that we only need to prove the statement about the negativity for $x > 1$, the result for $x \in (0,1)$ then follows immediately. We start with the following calculation;
	\begin{equation}\label{eq_log_log_concave}
		\begin{aligned}
			\left(x \frac{d}{dx}\right)^3 \log \left( \theta_3(x) \right) & = \left(x \, \frac{d}{dx}\right)^2 \left(x \, \psi(x) \right)\\
			& = x \, \psi(x) + 3 x^2 \, \psi'(x) + x^3 \, \psi''(x).
		\end{aligned}
	\end{equation}
	In particular, Proposition \ref{pro_theta3_symmetry} implies that $\psi(1) + 3 \psi'(1) + \psi''(1) = 0$. We proceed in two steps:
	\begin{enumerate}[(i)]
		\item{We will show that $\left(x \frac{d}{dx}\right)^3 \log \left( \theta_3(x) \right)$ is negative for $x > 1.1$.}\label{step1}
		\item{We will show that $\left(x \frac{d}{dx}\right)^3 \log \left( \theta_3(x) \right)$ is decreasing on an interval containing $\left( 1, 1.1 \right)$.}\label{step2}
	\end{enumerate}
	
	\textit{First step}: We will use a combination of the asymptotic and the local behavior to establish the result claimed in \eqref{step1}.

	Since $\psi$ is the logarithmic derivative of $\theta_3$, we use Jacobi's triple product formula for $\theta_3$ to obtain a series representation for $\psi$. In order to control the expression in equation \eqref{eq_log_log_concave} we compute the derivatives of $\psi$ up to order 2.
	\begin{align}
		\psi(x) & = \sum_{k \geq 1} \left( \frac{2 k \pi e^{-2 k \pi x}}{1-e^{-2 k \pi x}} - 2 \frac{(2 k -1) \pi e^{-(2 k -1) \pi x}}{1+e^{-(2 k -1) \pi x}} \right) = \frac{\theta_3'(x)}{\theta_3(x)}\\
		\psi'(x) & = \sum_{k \geq 1} \left( -\frac{(2 k \pi)^2 e^{-2 k \pi x}}{\left(1-e^{-2 k \pi x}\right)^2} + 2 \frac{((2 k -1) \pi)^2 e^{-(2 k -1) \pi x}}{\left(1+e^{-(2 k -1) \pi x}\right)^2} \right)\\
		\psi''(x) & = \sum_{k \geq 1} \left( \frac{(2 k \pi)^3 e^{-2 k \pi x}\left(1+e^{-2 k \pi x}\right)}{\left(1-e^{-2 k \pi x}\right)^3} - 2 \frac{((2 k -1) \pi)^3 e^{-(2 k -1) \pi x} \left(1 - e^{-(2k-1)\pi x}\right)}{\left(1+e^{-(2 k -1) \pi x}\right)^3} \right)
	\end{align}
	
	It is easy to verify that $\psi^{(n)}(x) = \mathcal{O}(e^{-\pi x})$, $n = 0,1,2$ (this actually holds for any $n \in \N$). However, due to the monomial term in $k$, which is of the order of the derivative plus 1, we will also consider the contributions of terms involving $e^{- 2 \pi x}$ to gain more control on the local behavior close to 1.
	
	The techniques are standard and the proof is inspired by the proofs in Montgomery's article \cite{Montgomery_Theta_1988}. In particular, we will estimate parts of the series by the leading term(s), by appropriate geometric series or by their values at $x = 1$ (or use combinations of the mentioned methods). We will try to indicate which kind of estimates are used at which point in the proof.
	
	We start with an estimate for $\psi$ for $x > 1$. We will first estimate the denominators and after that estimate the positive part by a geometric series and the negative part by the leading term. This will also be the common theme throughout the proof.
	\begin{align}
		\psi(x) & = \sum_{k \geq 1} \left( \frac{2 k \pi e^{-2 k \pi x}}{\underbrace{1-e^{-2 k \pi x}}_{\tfrac{1}{1.002}<}} - 2 \frac{(2 k -1) \pi e^{-(2 k -1) \pi x}}{\underbrace{1+e^{-(2 k -1) \pi x}}_{<\tfrac{1}{0.958}}} \right)\\
		& < 1.002 \cdot 2 \pi \sum_{k\geq 1} \left(k \, e^{-2 k \pi x} \right) - 0.958 \cdot 2 \pi \, e^{-\pi x}\\
		& = 1.002 \cdot 2 \pi \underbrace{\left(1-e^{-2 \pi x}\right)^{-2}}_{<1.004} \, e^{-2 \pi x} - 0.958 \cdot 2 \pi \, e^{-\pi x}\\
		& < 6.35 \, e^{-2\pi x} - 6 \, e^{- \pi x}
	\end{align}
		
	We proceed with an upper bound for $\psi'$ for $x > 1$.
	\begin{align}
		\psi'(x) & = \sum_{k \geq 1} \left( -\frac{(2 k \pi)^2 e^{-2 k \pi x}}{\underbrace{\left(1-e^{-2 k \pi x}\right)^2}_{<1}} + 2 \frac{((2 k -1) \pi)^2 e^{-(2 k -1) \pi x}}{\underbrace{\left(1+e^{-(2 k -1) \pi x}\right)^2}_{1<}} \right)\\
		& < -4 \pi^2 e^{-2 \pi x} + 2 \pi^2 \sum_{k \geq 1} \left( (2 k -1)^2 e^{-(2 k -1) \pi x} \right)\\
		& = -4 \pi^2 e^{-2 \pi x} + 2 \pi^2 \underbrace{\left(1 + 6 e^{-2 \pi x} + e^{-4 \pi x}  \right) \left( 1 - e^{-2 \pi x}\right)^{-3}}_{<1.017} \, e^{- \pi x} \\
		& < -39.4 \, e^{-2 \pi x} + 20.1 \, e^{-\pi x}.
	\end{align}
	In the second to last line we used the fact that (for $|q|<1$)
	\begin{align}
		\sum_{k \geq 1}(2k-1)^2 q^{2k-1} & = \sum_{k \geq 1} k^2 q^k - \sum_{k \geq 1} (2k)^2 q^{2k}\\
		& = q \frac{1+q}{(1-q)^3} - 4 q^2 \frac{1+q^2}{(1-q^2)^3} \\
		& = q \frac{1 + 6 q^2 + q^4}{(1-q^2)^3}.
	\end{align}
	The closed expressions for the according geometric series are obtained by taking derivatives of the classical geometric series and for the series involving only even term, $q$ is substituted by $q^2$.
	
	With the same techniques we bound $\psi''$ from above for $x > 1$.
		
	\begin{align}
		\psi''(x) & = \sum_{k \geq 1} \left( (2 k \pi)^3 e^{-2 k \pi x} \, \frac{\left(1+e^{-2 k \pi x}\right)}{\left(1-e^{-2 k \pi x}\right)^3} \right.\\
		& \qquad \qquad - \left. 2 \, ((2 k -1) \pi)^3 e^{-(2 k -1) \pi x} \, \frac{\left(1 - e^{-(2k-1)\pi x}\right)}{\left(1+e^{-(2 k -1) \pi x}\right)^3} \right)\\
		& < \sum_{k \geq 1} \left(\left(2 k \pi \right)^3 e^{-2 k \pi x}\underbrace{\left(1 + e^{-2\pi}\right)}_{< 1.002} \underbrace{\left(1 - e^{-2\pi}\right)^{-3}}_{< 1.006}\right)\\
		& \qquad \qquad
		- 2 \pi^3 e^{-\pi x} \left(\underbrace{\left( 1 - e^{-\pi x} \right) \left(1 + e^{-\pi x}\right)^{-3} + 27 \, e^{-2 \pi x} \left( 1 - e^{-3\pi x} \right) \left(1 + e^{-3\pi x}\right)^{-3}}_{>0.89}\right)\\
		& < 8.07 \, \pi^3 \, e^{-2 \pi x} \underbrace{\left( e^{- 4 \pi x}+4 e^{-2 \pi x} + 1 \right) \left(1-e^{-2 \pi x}\right)^{-4}}_{< 1.02} - 1.78 \, \pi^3 e^{-\pi x}\\
		& < 255.5 \, e^{-2 \pi x} - 55.1 \, e^{-\pi x}.
	\end{align}
	
	This time we truncated the negative series over the odd integers after 2 terms, as the monomial already has degree 3 and, therefore, the second term contributes some weight to the series evaluated near $x = 1$.
	
	Now, for $x > 1$ we have the following estimate;
	\begin{align}
		& \left(x \frac{d}{dx}\right)^3 \log \left( \theta_3(x) \right) = x \, \psi(x) + 3 x^2 \, \psi'(x) + x^3 \, \psi''(x)\\
		& < \left( 6.35 \, x - 118.2 \, x^2 + 255.5 \, x^3 \right) e^{-2 \pi x} + \left( -6 \, x +  60.3 \, x^2 - 55.1 \, x^3 \right) e^{-\pi x}.
	\end{align}
	
	It is obvious that the last expression asymptotically tends to zero from below, hence, there exists some $x_0 > 1$ such that
	\begin{equation}
		\left( x \dfrac{d}{dx} \right)^3 \log (\theta_3(x)) < 0, \qquad \forall x > x_0.
	\end{equation}
	
	We will now determine $x_0$. By the above estimates, it is enough to show that the following inequality holds for $x > x_0 > 1$;
	\begin{equation}\label{eq_x0}
		\left( 6.35 - 118.2 \, x + 255.5 \, x^2 \right) e^{-\pi x} + \left( -6 + 60.3 \, x - 55.1 \, x^2 \right) < 0
	\end{equation}
	It is not hard to check that $-6 + 60.3 \, x - 55.1 \, x^2$ is strictly decreasing for $x > 1$. Also, we note that expressions of the form $\left( a + b \, x + c \, x^2\right) e^{-\pi x}$ are strictly decreasing for
	\begin{equation}
		x > \frac{2c - \pi b + \sqrt{\pi^2 b^2 - 4 \pi^2 a c + 4 c^2}}{2 \pi c}.
	\end{equation}
	For $a = 6.35$, $b = -118.2$ and $c = 255.5$, we find out that for $x>1$ the expression is indeed strictly decreasing. Therefore, it suffices to find a value $x_0$ such that \eqref{eq_x0} is true, as it then holds for all $x > x_0$. For $x = 1.1$ the left-hand side of \eqref{eq_x0} is smaller than $-0.4$ and we found the desired value $x_0 = 1.1$.
	
	\textit{Second step}: We will now show that $\left(x \frac{d}{dx}\right)^3 \log \left( \theta_3(x) \right)$ is strictly decreasing on an interval containing $\left( 1, x_0 \right)$. To do so, we apply the differential operator $x \frac{d}{dx}$ on \eqref{eq_log_log_concanve_prop}, i.e.,
	\begin{equation}\label{eq_dx4}
		\begin{aligned}
			\left(x \frac{d}{dx}\right)^4 \log \left( \theta_3(x) \right) & = x \frac{d}{dx} \left(x \, \psi(x) + 3 x^2 \, \psi'(x) + x^3 \, \psi''(x)\right)\\
			& = x \, \psi(x) + 7 x^2 \, \psi'(x) + 6 x^3 \, \psi''(x) + x^4 \, \psi'''(x).
		\end{aligned}
	\end{equation}
	
	In order to control the last expression for $x \in(1,x_0)$, we need good estimates on $\psi'''$ on this interval. We start in the same manner as for the lower order derivatives.
	
	\begin{align}
		\psi'''(x) & = \sum_{k \geq 1} \left(
		  -(2 k \pi)^4 e^{-2 k \pi x} \frac{\left(e^{-4 k \pi x} + 4 e^{-2 \pi  k x} + 1 \right)}{\left(1 - e^{-2 k \pi x}\right)^4}
		\right.\\
		& \qquad \qquad \left.
		  + 2 \, ((2k-1) \pi)^4 \, e^{-(2k-1)\pi x} \frac{\left(e^{-2 (2k-1)\pi x} - 4 e^{-(2k-1)\pi x} + 1\right)}{\left(1 + e^{-(2k-1)\pi x}\right)^4}
		\right)\\
		& < -16 \, \pi ^4 e^{-2 \pi x} \underbrace{(e^{-4 \pi x} + 4 e^{-2 \pi x} + 1) \left(1- e^{-2 \pi x}\right)^{-4}}_{> 1}\\
		& \qquad \qquad + 2 \pi^4 \sum_{k \geq 1} \left( (2k-1)^4 \, e^{-(2k-1)\pi x} \left(1 - e^{-(2k-1)\pi x}\right)^2 \left(1 + e^{-(2k-1)\pi x}\right)^{-4}\right).
	\end{align}
	Now, we make use of the fact that we only need to establish that the expression in \eqref{eq_dx4} is negative on the interval $(1,x_0)$. Due to the continuity of the function, this estimate then also holds on a slightly larger interval, which we do not have to specify. We estimate that, on $(1, x_0)$,
	\begin{equation}
		\left(1 - e^{-(2k-1)\pi x}\right)^2 \left(1 + e^{-(2k-1)\pi x}\right)^{-4} < \left(1 - e^{-1.1 \cdot \pi}\right)^2 \left(1 + e^{-1.1 \cdot \pi}\right)^{-4} < 0.83, \quad k = 1.
	\end{equation}
	As $k \to \infty$, the value of the left-hand side of the above inequality tends to 1 from below, which gives the uniform estimate
	\begin{equation}
		\left(1 - e^{-(2k-1)\pi x}\right)^2 \left(1 + e^{-(2k-1)\pi x}\right)^{-4} < 1, \quad \forall k \geq 2.
	\end{equation}
	Therefore, on the interval $(1,x_0)$ we get the estimate
	\begin{align}
		\psi'''(x) & < - 16 \, \pi ^4 \, e^{-2 \pi x} + 1.66 \, \pi^4 \, e^{-\pi x} + 2 \, \pi^4 \sum_{k \geq 2} (2k-1)^4 \, e^{-(2k-1) \pi x}.
	\end{align}
	We now use the closed expression for the following geometric series, established in the same manner as in the estimates for $\psi''$;
	\begin{equation}\label{eq_series_q}
		\sum_{k \geq 2} (2k-1)^4 q^{2k-1} = q \, \frac{81 q^2 +220 q^4 + 86 q^6 - 4 q^8 + q^{10}}{\left(1 - q^ 2 \right)^5}, \qquad |q|<1.
	\end{equation}
	For $q = e^{- \pi x}$ and $x \in (1, x_0)$ we have that
	\begin{equation}
		\frac{81 q^2 +220 q^4 + 86 q^6 - 4 q^8 + q^{10}}{\left(1 - q^ 2 \right)^5} < \left. \frac{81 q^2 +220 q^4 + 86 q^6 + q^{10}}{\left(1 - q^ 2 \right)^5} \right|_{q = e^{-\pi}} < 0.16.
	\end{equation}
	Therefore, we get that, on the interval $(1,x_0)$,
	\begin{align}
		\psi'''(x) & < -16 \, \pi^4 \, e^{-2 \pi x} + 1.98 \, \pi^4 e^{-\pi x} \\
		& < -1558 \, e^{-2 \pi x} + 193 \, e^{- \pi x}.
	\end{align}
	In total, it follows that for $x \in (1,x_0)$ the following inequality holds;
	\begin{equation}
		\begin{aligned}
			\left(x \frac{d}{dx}\right)^4 \log \left( \theta_3(x) \right) & = x \, \psi(x) + 7 x^2 \, \psi'(x) + 6 x^3 \, \psi''(x) + x^4 \psi'''(x)\\
			& < \left(6.35 \, x - 275.8 \, x^2 + 1533 \, x^3 - 1558 \, x^4 \right) e^{- 2 \pi x}\\
			& \quad + \left(-6 \, x + 140.7 \, x^2 - 330.6 \, x^3 + 193 \, x^4 \right) e^{-\pi x}.
		\end{aligned}
	\end{equation}
	We need to verify that
	\begin{equation}
		\left(6.35 \, x - 275.8 \, x^2 + 1533 \, x^3 - 1558 \, x^4 \right) e^{- 2 \pi x} + \left(-6 \, x + 140.7 \, x^2 - 330.6 \, x^3 + 193 \, x^4 \right) e^{-\pi x} < 0
	\end{equation}
	for $x \in (1,x_0)$. This is equivalent to showing that
	\begin{equation}
		\left(6.35 - 275.8 \, x + 1533 \, x^2 - 1558 x^3 \right) e^{- \pi x} + \left(-6 + 140.7 \, x - 330.6 x^2 + 193 x^3 \right) < 0
	\end{equation}
	In order to establish this result, we note that
	\begin{align}
		\left(6.35 - 275.8 \, x + 1533 \, x^2 - 1558 x^3 \right) < 0, \qquad x > 1.
	\end{align}
	Hence, we get the estimate
	\begin{align}
		\left(6.35 - 275.8 \, x + 1533 \, x^2 - 1558 \, x^3 \right) e^{- \pi x} & < \left(6.35 - 275.8 \, x + 1533 \, x^2 - 1558 \, x^3 \right) e^{- 1.1 \cdot \pi}\\
		& < 0.3 - 8.7 \, x + 48.6 \, x^ 2 - 49 \, x^3,
	\end{align}
	for $x \in (1,x_0)$. Therefore, by showing the truth of the (stronger) inequality
	\begin{equation}\label{eq_final}
		-5.7 + 132 \, x - 282 \, x^2 + 144 \, x^3 < 0, \quad x \in (1,x_0),
	\end{equation}
	we can finish the proof. It is not hard to show that the left-hand side of the above inequality has critical points at $\frac{11}{36} \notin (1, x_0)$ (local maximum) and $1$ (local minimum). Finally, by checking the values on the boundary of the interval $(1, x_0)$ we see that \eqref{eq_final} indeed holds.
	
	By combining the two steps we see that the function $\left(x \frac{d}{dx}\right)^3 \log \left( \theta_3(x) \right)$ is strictly decreasing at least on the interval $(1, 1.1)$ and negative for $x > 1.1$. As the value at $x = 1$ is zero, we can finally conclude that the expression given in equation \eqref{eq_log_log_concanve_prop} is negative for $x > 1$. Due to the already mentioned anti--symmetry with respect to the point $(1,0)$ the function has to be positive for $0 < x <1$.
\end{proof}

Vesalainen and Ernvall--Hytönen \cite{HytVes_Personal}, \cite{HytVes_Secrecy_2017} proved Proposition \ref{prop_log_log_concave} independently from this work. They used the result to prove a conjecture by Hernandez and Sethuraman\footnote{The author does not claim any credit for solving the conjecture by Hernandez and Sethuraman since the author was not aware of the conjecture prior to the appearance of \cite{HytVes_Secrecy_2017}.} \cite{Hernandez_Master_2016} (see also the references in \cite{HytVes_Secrecy_2017}). The following result by Vesalainen and Ernvall--Hytönen is an immediate consequence of Proposition \ref{prop_log_log_concave}.
\begin{theorem}[Ernvall--Hytönen, Vesalainen]
	For $\alpha, \beta \in \R_+$ with $1 \leq \alpha < \beta$ and $x \in \R_+$ we define the function
	\begin{equation}
		g(x) = \frac{\theta_3(\beta x) \theta_3(\tfrac{x}{\beta})}{\theta_3(\alpha x) \theta_3(\tfrac{x}{\alpha})}.
	\end{equation}
	Then
	\begin{equation}
		g(x) = g(\tfrac{1}{x})
	\end{equation}
	and $g$ assumes its global maximum only for $x = 1$. Also, the function is strictly increasing for $x \in (0,1)$ and strictly decreasing for $x \in (1,\infty)$.
\end{theorem}

\subsection{A Result Involving Theta-1 and a Proof of Theorem \ref{thm_main}}\label{sec_Theta_1}
In contrast to the fact that $\vartheta_1(0,\tau) = 0$ for all $\tau \in \mathbb{H}$, its derivative with respect to $z$ evaluated at $(0,i x)$, $x \in \R_+$ is not the zero function, in particular
\begin{equation}\label{eq_theta_1}
	\begin{aligned}
		\partial_z \vartheta_1(0,i x) & = \sum_{k \in \Z} (-1)^k (2k+1) e^{-\pi (k+1/2)^2 x}\\
		& = 2 \, e^{-\tfrac{\pi}{4} x} \prod_{k \in \N} \left(1 - e^{-2 k \pi x}\right)^3, \quad x \in \R_+.
	\end{aligned}
\end{equation}
The above function is often denoted by $\theta_1'(x)$ (see e.g., \cite{ConSlo98}), where the prime indicates differentiation with respect to $z$. The representation as an infinite product is also well--known and follows from the Jacobi triple product representation of $\Theta$ (Proposition \ref{prop_triple_product}). Since the notation $\theta_1'$ could be misleading in this work, we rather use the notation from \cite{Borwein_AGM_1987} and set
\begin{equation}
	\theta_1^+(x) = \partial_z \vartheta_1(0, i x), \quad x \in \R_+.
\end{equation}
Now we have the following identity (see e.g., \cite[chap.~3]{Borwein_AGM_1987} or \cite[Chap.~4]{ConSlo98})
\begin{equation}\label{eq_identity_theta_1}
	\theta_2(x) \theta_3(x) \theta_4(x) = \theta_1^+(x).
\end{equation}
From this identity we derive the following result.
\begin{theorem}\label{thm_psi}
	For $x \in \R_+$, we define the function
	\begin{equation}
		\psi_1(x) = x^{3/4} \theta_1^+(x).
	\end{equation}
	This function also has the following representations
	\begin{equation}
		\psi_1(x) = x^{3/4} \prod_{j=2}^4 \theta_j(x) = 2 \, x^{3/4} e^{-\tfrac{\pi}{4} x} \prod_{k \in \N} \left(1 - e^{-2 k \pi x}\right)^3.
	\end{equation}		
	The function is positive and assumes its global maximum only for $x = 1$ and $x \frac{\psi_1'(x)}{\psi_1(x)}$ is strictly decreasing. Furthermore,
	\begin{equation}
		\psi_1(x) = \psi_1(\tfrac{1}{x})
	\end{equation}	
	and, hence,
	\begin{equation}
		\left(x \frac{d}{dx} \right)^n \psi_1(x) = (-1)^n \left(x \frac{d}{dx} \right)^n \psi_1(\tfrac{1}{x}), \quad n \in \N.
	\end{equation}
\end{theorem}

\begin{proof}
	The different representations of $\psi_1$ follow immediately from \eqref{eq_theta_1} and \eqref{eq_identity_theta_1}. It is obvious from the product representation that $\psi_1$ is positive. The symmetry of $\psi_1$ follows easily from \eqref{eq_Poisson_theta_3} and \eqref{eq_Poisson_theta_2_4}. The (anti--)symmetry for the repeated application of the differential operator $x \frac{d}{dx}$ immediately follows by induction.
	Differentiating $\log(\psi_1)$ on a logarithmic scale, i.e., applying the differential operator $x \frac{d}{dx}$, yields
	\begin{equation}
		x \frac{d}{dx} \log(\psi_1(x)) = \frac{3}{4} + \sum_{j=2}^4x \frac{\theta_j'(x)}{\theta_j(x)}.
	\end{equation}
	In particular, this means that $\psi_1$ has a critical point at $x = 1$ since
	\begin{equation}
		\sum_{j=2}^4 \frac{\theta_j'(1)}{\theta_j(1)} = -\frac{3}{4},
	\end{equation}
	which follows from \eqref{eq_identity_theta_3} and \eqref{eq_identity_theta_2_4}. The fact that $\psi_1$ assumes its global maximum only for $x = 1$ will follow from the fact that $x \frac{\psi_1'(x)}{\psi_1(x)}$ is strictly decreasing because, since $x > 0$ and $\psi_1(x) > 0$ we have that
	\begin{equation}
		\psi_1'(x) < 0 \Leftrightarrow x \frac{\psi_1'(x)}{\psi(x)} < 0.
	\end{equation}		
	Since $\frac{\psi_1'(1)}{\psi_1(1)} = 0$, the negativity of $\psi_1'$ for $x > 1$ already follows if we can show that $x \frac{\psi_1'(x)}{\psi_1(x)}$ is decreasing. We use the infinite product representation to establish this result.
	\begin{equation}
		\begin{aligned}
			x \frac{d}{dx} \log(\psi_1(x)) & = x \frac{d}{dx} \log \left(2 \, x^{3/4} e^{-\tfrac{\pi}{4} x} \prod_{k \in \N} \left(1 - e^{-2 k \pi x}\right)^3 \right)\\
			& = x \frac{d}{dx} \left( \log(2)  +\tfrac{3}{4} \log(x) - \tfrac{\pi}{4}x + 3 \sum_{k \in \N} \log \left(1 - e^{-2 k \pi x}\right) \right)\\
			& = \tfrac{3}{4} - \tfrac{\pi}{4} x + 3 \sum_{k \in \N} \frac{2 k \pi x}{e^{2 k \pi x} - 1}.
		\end{aligned}
	\end{equation}
	It is quickly verified that, except for the constant term, all terms in this expression are decreasing and the proof is finished.
\end{proof}

\begin{figure}[ht]
	\includegraphics[width=.45\textwidth]{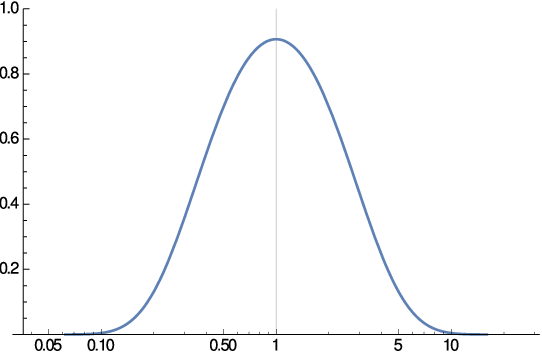}
	\hfill
	\includegraphics[width=.45\textwidth]{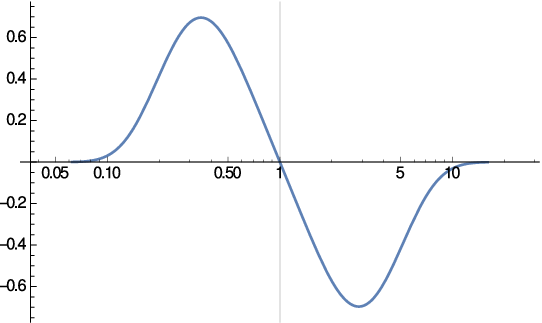}
	\caption{The function $\psi_1(x)$ and its derivative $x \frac{d}{dx} \psi_1(x)$ plotted on a logarithmic scale.}\label{fig_psi}
\end{figure}

\begin{figure}[ht]
	\includegraphics[width=.45\textwidth]{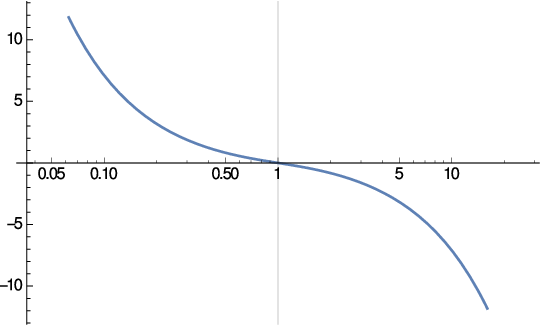}
	\hfill
	\includegraphics[width=.45\textwidth]{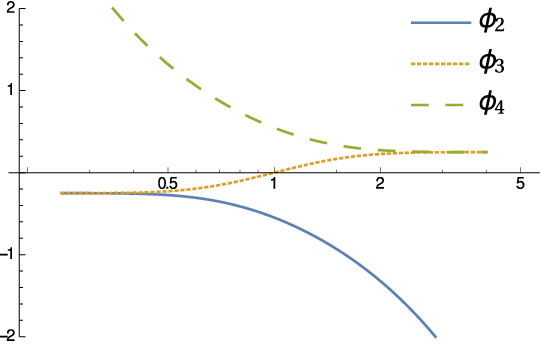}
	\caption{The logarithmic derivative $x \frac{d}{dx} \log(\psi_1(x))$ plotted on a logarithmic scale and split up in its components $\phi_j(x) = x \frac{\theta_j'(x)}{\theta_j(x)}+\tfrac{1}{4}$, $j=2,3,4$.}\label{fig_log}
\end{figure}

Figure \ref{fig_psi} shows the function $\psi_1(x)$ and its derivative with logarithmic scaling on the $x$--axis, revealing the symmetry properties. In Figure \ref{fig_log} we see the logarithmic derivative of $\psi_1$ on a logarithmic scale, i.e., $x \tfrac{d}{dx} \log (\psi_1(x))$. The behavior of this function is strongly influenced by the above established properties of Jacobi's theta functions.

Finally, we observe that, for $x \in \R_+$,
\begin{equation}
	\psi_1(x) = 2 x^{3/4} |\eta(i x)|^3,
\end{equation}
which follows from equation \eqref{eq_eta}. Theorem \ref{thm_main} now follows immediately from Theorem \ref{thm_psi}.

\end{document}